\documentclass[12pt]{article}

\usepackage{setspace}

\usepackage{amsmath,amssymb,amsthm}
\usepackage[latin1]{inputenc}
\usepackage{color}

\begin{document}

\newtheorem{theorem}{Theorem}
\newtheorem{proposition}[theorem]{Proposition}
\newtheorem{remark}[theorem]{Remark}
\newtheorem{lemma}[theorem]{Lemma}
\newtheorem{corollary}[theorem]{Corollary}
\newtheorem{definition}[theorem]{Definition}

\newcommand{\R}{\mathbb{R}}
\newcommand{\N}{\mathbb{N}}
\newcommand{\PP}{\mathbb{P}}
\newcommand{\EE}{\mathbb{E}}
\newcommand{\E}{\mathbb{E}}
\newcommand{\wn}{\widehat{\nabla}}
\newcommand{\eps}{\varepsilon}

\newcommand{\dps}{\displaystyle}
\newcommand{\sgn}{\text{sgn}}

\title{Pathwise estimates for an effective dynamics}
\author{F. Legoll$^{1,4}$, T. Lelièvre$^{2,4}$ and S. Olla$^3$\\
{\footnotesize $^1$ Laboratoire Navier, \'Ecole des Ponts ParisTech, Universit\'e Paris-Est,}\\
{\footnotesize 6 et 8 avenue Blaise Pascal, 77455 Marne-La-Vall\'ee Cedex 2, France}\\
{\footnotesize \tt legoll@lami.enpc.fr}\\
{\footnotesize $^2$ CERMICS, \'Ecole des Ponts ParisTech, Universit\'e Paris-Est,}\\
{\footnotesize 6 et 8 avenue Blaise Pascal, 77455 Marne-La-Vall\'ee Cedex 2, France}\\
{\footnotesize \tt lelievre@cermics.enpc.fr}\\
{\footnotesize $^3$ CEREMADE, UMR-CNRS 7534, Université de Paris Dauphine, PSL Research University}\\
{\footnotesize Place du Maréchal De Lattre De Tassigny, 75775 Paris Cedex 16, France}\\
{\footnotesize \tt olla@ceremade.dauphine.fr}\\
{\footnotesize $^4$ INRIA Paris, MATHERIALS project-team,}\\
{\footnotesize 2 rue Simone Iff, CS 42112, 75589 Paris Cedex 12, France}
}
\date{\today}

\maketitle

\begin{abstract}
Starting from the overdamped Langevin dynamics in $\R^n$,
$$
dX_t = -\nabla V(X_t) dt + \sqrt{2 \beta^{-1}} dW_t,
$$
we consider a scalar Markov process $\xi_t$ which approximates the dynamics of the first component $X^1_t$. In the previous work~\cite{legoll-lelievre-10}, the fact that $(\xi_t)_{t \ge 0}$ is a good approximation of $(X^1_t)_{t \ge 0}$ is proven in terms of time marginals, under assumptions quantifying the timescale separation between the first component and the other components of $X_t$. Here, we prove an upper bound on the trajectorial error $\dps \EE \left( \sup_{0 \leq t \leq T} \left| X^1_t - \xi_t \right| \right)$, for any $T > 0$, under a similar set of assumptions. We also show that the technique of proof can be used to obtain quantitative averaging results.
\end{abstract}

\section{Introduction}


Coarse-graining techniques are fundamental tools in computational statistical physics problems. They are very important for modelling questions (to get some insight on a complicated high-dimensional problem, by reducing it to a low-dimensional model) and for numerical algorithms, which very often use coarse-grained descriptions as predictors to speed up the computations.

In this work, we are interested in getting a low-dimensional Markov dynamics on a few degrees of freedom starting from a high-dimensional Markov dynamics. More specifically, we consider a stochastic process $(X_t)_{t \ge 0}$ on $\R^n$ which follows the overdamped Langevin dynamics:
\begin{equation}
\label{eq:X}
dX_t = - \nabla V(X_t) \, dt + \sqrt{2 \beta^{-1}} \, d W_t,
\end{equation}
where $\beta^{-1}$ is proportional to the temperature, $W_t$ is a standard $n$-dimensional Brownian motion and $V : \R^n \to \R$ is a smooth function. This dynamics is often used in molecular dynamics simulation. Under suitable assumptions on~$V$, this dynamics is ergodic with respect to the Boltzmann-Gibbs measure
$$
d\mu = \psi(x) \, dx
$$
with
$$
\psi(x)
=
Z^{-1} \exp(-\beta V(x)),
\quad
Z=\int_{\R^n} \exp(-\beta V(x)) \, dx,
$$
where $Z$ is assumed to be finite. Suppose that we are not interested in the dynamics of $(X_t)_{t \ge 0}$, but only in the dynamics of its first component, $(X^1_t)_{t \ge 0}$ (see Remark~\ref{rem:xi} below for a discussion of more general cases). In view of~\eqref{eq:X}, we have
\begin{equation}
\label{eq:X1}
dX^1_t = - \partial_1 V(X_t) \, dt + \sqrt{2 \beta^{-1}} \, d W^1_t.
\end{equation}
In general, this dynamics is not closed in $(X^1_t)_{t \ge 0}$, as the right-hand side depends on $X_t$ and not only on $X^1_t$. 

To obtain a closed (Markov) dynamics, a natural idea inspired by~\cite{gyongy-86} is to replace the drift term in~\eqref{eq:X1} by its expectation with respect to $\mu$, conditionally to the value of $X^1_t$. We refer to~\cite{legoll-lelievre-10,legoll-lelievre-12} for a motivation using probabilistic arguments, and to~\cite{e-vanden-eijnden-04,maragliano-fischer-vanden-einjden-ciccotti-06} for other derivations, based in particular on the so-called Mori-Zwanzig projection operator approach~\cite{givon-kupferman-stuart-04}. Roughly speaking, such an approximation requires the first component $X^1_t$ to ``move slowly'' compared to the other components $(X^2_t, \ldots, X^n_t)$, so that these components reach equilibrium before $X^1_t$ has moved. The aim of the analysis performed here and in the previous works~\cite{legoll-lelievre-10,legoll-lelievre-12} is to give a precise mathematical content to this intuition.

For any $\xi \in \R$, we hence introduce
\begin{align}
\label{eq:def_b}
b(\xi) 
&=
\EE_\mu\Big( \partial_1 V(X) \, \Big| \, X^1 = \xi \Big)
=
\int_{\R^{n-1}} \partial_1 V(\xi,x_2^n) \ \psi^\xi(x_2^n) \, dx_2^n
\end{align}
with
\begin{equation}
\label{eq:def_psixi}
\psi^\xi(x_2^n) 
=
\frac{\psi(\xi,x_2^n) 
}{
\dps \int_{\R^{n-1}} \psi(\xi,x_2^n) 
\, dx_2^n
},
\end{equation}
where we have used the notation
$$
x_2^n = (x^2,\dots,x^n), \quad dx_2^n = dx^2 \dots dx^n.
$$
Note that $\psi^\xi(x_2^n) \, dx_2^n$ is a conditional probability measure, namely the probability measure $\mu$ conditioned to $X^1=\xi$. In the following, we also need the marginal probability density $\varphi:\R \to \R$ of $\psi$ along the first coordinate:
\begin{equation}
\label{eq:def_varphi}
\varphi(\xi) = \int_{\R^{n-1}} \psi(\xi,x_2^n) \, dx_2^n.
\end{equation}

The function $b$ (or, depending on the authors, $-b$) is the so-called {\em mean force} associated to the measure $\mu$ and the {\em reaction coordinate} $(x^1,x_2^n) \mapsto x^1$, see e.g.~\cite{lelievre-rousset-stoltz-book-10}. It is the derivative of the so-called {\em free energy} $F(\xi)=-\beta^{-1} \ln \varphi(\xi)$ associated to $\mu$ and the reaction coordinate $(x^1,x_2^n) \mapsto x^1$:
$
b(\xi) = F'(\xi).
$

Replacing the drift term in~\eqref{eq:X1} by its conditional expectation, we introduce the following dynamics, which hopefully is a good approximation of~\eqref{eq:X1}:
\begin{equation}
\label{eq:dyn_eff}
\left\{
\begin{aligned}
d\xi_t &= - b(\xi_t) \, dt + \sqrt{2 \beta^{-1}} \, dW^1_t,\\
\xi_{0} &= X^1_0.
\end{aligned}
\right.
\end{equation}
This is a closed dynamics in $\xi_t$, and $(\xi_t)_{t \geq 0}$ is a Markov process. The mathematical question is now to estimate, in some sense to be made precise, the distance between $\xi_t$ solution to~\eqref{eq:dyn_eff} and $X^1_t$ which satisfies~\eqref{eq:X1}--\eqref{eq:X}. By construction, the effective dynamics~\eqref{eq:dyn_eff} has the correct stationary state: it is ergodic with respect to $\varphi(\xi)\,d\xi$, which is precisely the law of $X^1_t$ in the longtime limit. The question we address here concerns the correctness of the {\em dynamics}. This is motivated in particular by current practices in the field of molecular dynamics, where practitioners derive from~\eqref{eq:dyn_eff} some transition times by looking at free energy differences and using the Eyring-Kramers law (see for example~\cite{hanggi-talkner-barkovec-90} for a review).

\medskip

As a first step, in~\cite{legoll-lelievre-10}, estimates on the distance (in total variation norm) between the law at time $t$ of $X^1_t$ and the law at time $t$ of $\xi_t$ have been obtained. These are therefore estimates on the distance between the marginals in time. More precisely, under the two following assumptions:
\begin{itemize}
\item[{\bf[A1]}] The conditional probability measures $\psi^\xi(x_2^n) \, dx_2^n$ satisfy a Logarithmic Sobolev Inequality with a constant $\rho$ which does not depend on~$\xi$,
\item[{\bf [A2]}] The so-called coupling constant $\kappa_\infty$ is finite: 
$$
\kappa_\infty = \|\partial_1 \wn V \|_{L^\infty(\R^n)} < \infty,
$$
where $\wn V = (\partial_2 V, \ldots, \partial_n V)$,
\end{itemize}
we have shown in~\cite{legoll-lelievre-10} (see also~\cite{legoll-lelievre-12} for a simple case) that, for all times $t \ge 0$,
\begin{equation}
\label{eq:non}
H \Big( {\mathcal L}(X^1_t) \, \Big| \, {\mathcal L}(\xi_t) \Big) 
\le 
\frac{\beta^2 \kappa_\infty^2}{4 \rho^2} \left[
H \Big({\mathcal L}(X_0) \, \Big| \, \mu \Big) - 
H \Big({\mathcal L}(X_t) \, \Big| \, \mu \Big) 
\right].
\end{equation}
Here, $\dps H \Big({\mathcal L}(X^1_t) \, \Big| \, {\mathcal L}(\xi_t)
\Big)=\int_{\R} \ln\left( \frac{d {\mathcal L}(X^1_t)}{ d {\mathcal L}(\xi_t)} \right) d {\mathcal L}(X^1_t)$ 
denotes the relative entropy of the law of $X^1_t$ with respect to the law of $\xi_t$, which is for example an upper bound on the square distance between the two laws in total variation norm. Provided that $\rho$ is large, we hence see that, in terms of laws at any time $t$, $\xi_t$ is an accurate approximation of $X^1_t$. The assumption that $\rho$ is large formalizes the fact that mixing with respect to the probability measures $\psi^\xi(x_2^n) \, dx_2^n$ is fast, and hence, as pointed out above, that the components $(X^2_t, \ldots, X^n_t)$ quickly reach equilibrium.

As a side remark, let us mention that the set of assumptions [A1]--[A2] appears to be very useful to analyze metastable processes and coarse-graining techniques in many contexts (see e.g.~\cite{grunewald-otto-villani-westdickenberg-09,lelievre-09,lelievre-rousset-stoltz-08}).

\medskip

In this work, we go further and obtain estimates on some distance between {\em the trajectories} $(X^1_t)_{t \ge 0}$ and $(\xi_t)_{t \ge 0}$ (and not only the laws of $X^1_t$ and $\xi_t$ at any time $t$), under assumptions very similar to Assumptions [A1]--[A2] (see Proposition~\ref{prop:main_global} below). The main difference is that we obtain a result over finite time intervals, whereas~\eqref{eq:non} is a uniform in time estimate. Getting some trajectorial estimates is crucial since, in molecular dynamics, many quantities of interest are indeed trajectorial ones (such as autocorrelation in time of some observables, for example).

\medskip

The article is organized as follows. In Section~\ref{sec:not}, we introduce some notation, the assumptions under which we work, and our main result (Proposition~\ref{prop:main_global}). This result is based on three ingredients:
\begin{itemize}
\item first, an estimate on $ \EE_{\mu}( f^2 )$ where $f=b - \partial_1 V$, which is a direct consequence of two assumptions similar to the assumptions [A1] and [A2] above, see Section~\ref{sec:f};
\item second, the introduction of a Poisson equation and the use of an argument due to T.~Lyons and T.~Zhang in~\cite{lyons-zhang-94} to get an estimate on $\dps \EE_{\mu} \left( \sup_{0 \leq t \leq T} \left|\int_0^t f(X_s) ds \right|^2 \right)$, see Section~\ref{sec:LZ};
\item third, a Gronwall type argument to deduce from this estimate a bound on $\dps \EE \left( \sup_{0 \leq t \leq T} \left| X^1_t - \xi_t \right| \right)$. This argument requires some Lipschitz type assumptions on $b$, see Section~\ref{sec:gronwall}.
\end{itemize} 
Section~\ref{sec:conc} collects all these results to conclude the proof of Proposition~\ref{prop:main_global}. Finally, as an application of the mathematical techniques used to prove Proposition~\ref{prop:main_global}, we provide in Section~\ref{sec:homog} a quantitative averaging result, which is to the best of our knowledge new since it does not require the effective drift $b$ to be Lipschitz.

\begin{remark}
\label{rem:xi}
We consider here for simplicity the case when the degree of freedom of interest is one of the cartesian coordinate of $X_t$. This could be generalized in two directions. First, one could consider a more general function of $X_t$, say $\theta(X_t)$, where $\theta: \R^n \to \R$. This is actually the setting of the previous work~\cite{legoll-lelievre-10}. One could also consider higher dimensional settings, where $\theta$ takes its values in $\R^d$ with $d \ge 2$. We do not pursue along these directions here, in order to keep the presentation simple.
\end{remark}


\section{Notation, assumptions and main result}
\label{sec:not}

\subsection{Notation}

Let us introduce the operator $L$, defined by: for any function $v:\R^n \to \R$,
\begin{align}
\nonumber
L v &= 
- \nabla V \cdot \nabla v 
+ \beta^{-1} \Delta v
\\
\label{eq:def_L}
&= - \sum_{i=1}^n \partial_i V \, \partial_i v
+ \beta^{-1} \sum_{i=1}^n \partial_{ii} v.
\end{align}
We also need the family of operators $L^\xi$ indexed by $\xi \in \R$ and defined by: for any function $v:\R^n \to \R$,
\begin{align}
\nonumber
(L^\xi v)(\xi,x_2^n) & =
- \wn V(\xi,x_2^n) \cdot \wn v(\xi,x_2^n)
+ \beta^{-1} \widehat{\Delta} v(\xi,x_2^n)
\\ 
\label{eq:def_Lxi}
& = 
- \sum_{i=2}^n \partial_i V(\xi,x_2^n) \,
\partial_i v(\xi,x_2^n)
+ \beta^{-1} \sum_{i=2}^n \partial_{ii} v(\xi,x_2^n),
\end{align}
where we used the notation
$$
\widehat{\nabla} v = (\partial_2 v, \ldots, \partial_n v)
\ \ \text{ and } \ \ 
\widehat{\Delta} v = \sum_{i=2}^n \partial_{ii} v.
$$
Note that the sums in~\eqref{eq:def_Lxi} start at $i=2$, in contrast with those in~\eqref{eq:def_L}.

The functional space
$$
L^2(\psi)=\left\{ v:\R^n \to \R, \ v \in L^1_{\rm loc}(\R^n) \text{ and } \int_{\R^n} v^2 \psi < \infty \right\}
$$
plays a crucial role in the following. It is an Hilbert space for the scalar product: for $u$ and $v$ in $L^2(\psi)$,
$$
\langle u , v \rangle_{\psi} = \int_{\R^n} u \, v \, \psi.
$$
Likewise, we will use the space $L^2(\psi^\xi)$, defined over functions $v:x_2^n \in \R^{n-1} \mapsto v(x_2^n) \in \R$.

For a given operator $O$, we denote by $O^\star$ its adjoint with respect to the scalar product in $L^2(\psi)$: for any smooth test functions $u$ and $v$,
$$
\langle O^\star u , v \rangle_{\psi} = \langle u , O v \rangle_{\psi}.
$$
It is standard to check that $L$ is a symmetric operator in $L^2(\psi)$ (which is equivalent to the reversibility of the process $X_t$ with respect to the equilibrium measure $\mu$): for any smooth test functions $u$ and $v$,
$$
\langle L u , v \rangle_{\psi} = \langle u , L v \rangle_{\psi} = - \beta^{-1} \int_{\R^n} \nabla u \cdot \nabla v \ \psi.
$$
We thus have $L^\star=L$, and $L=- \beta^{-1} \nabla^\star \nabla$. 

\subsection{Assumptions}

In the sequel, we work under the three following assumptions.

First, we assume that, for any $\xi$, the conditional probability measures $\psi^\xi(x_2^n) \, dx_2^n$ defined by~\eqref{eq:def_psixi} satisfy a Poincaré inequality for a constant $\rho$ independent of $\xi$: there exists $\rho >0$ such that, for any $\xi$ and any function $v \in H^1(\psi^\xi)$, it holds:
\begin{multline}
\label{eq:poincare}
\int_{\R^{n-1}} \left(v(x_2^n) - \int_{\R^{n-1}} v(x_2^n) \, \psi^\xi(x_2^n) \, dx_2^n\right)^2 \, \psi^\xi(x_2^n) \, dx_2^n 
\\
\le \frac{1}{\rho} \int_{\R^{n-1}} \left| \wn v(x_2^n) \right|^2 \, 
\psi^\xi(x_2^n) \, dx_2^n,
\end{multline}
where the functional space $H^1(\psi^\xi)$ is defined by
\begin{equation}
\label{eq:H1}
H^1(\psi^\xi)=\left\{v:\R^{n-1} \to \R, \ \int_{\R^{n-1}} \left( v^2 + \big|\wn v\big|^2 \right) \psi^\xi < \infty \right\}.
\end{equation}
Note that, by integration by parts, we have
$$
\int_{\R^{n-1}} (-L^\xi \phi)(x_2^n) \, \phi(x_2^n) \, 
\psi^\xi(x_2^n) \, dx_2^n
=
\beta^{-1} \int_{\R^{n-1}} \left| \wn \phi(x_2^n) \right|^2 \, 
\psi^\xi(x_2^n) \, dx_2^n
$$
for any $\phi \in {\mathcal C}^\infty_0(\R^{n-1})$. The assumption~\eqref{eq:poincare} is thus a spectral gap assumption on the operator $-L^\xi$. A Poincaré inequality holds on a probability measure $\exp(-\beta W(x)) \, dx$ under relatively mild assumption on $W$. Basically, if $W$ is smooth and grows at least linearly at infinity, then $\exp(-\beta W(x)) \, dx$ satisfies a Poincaré inequality (see for example~\cite{ABC-00}). In particular, if $W$ is $\alpha$-convex, then the Poincaré inequality is satisfied with the constant $\alpha/2$.

\medskip

Second, we assume that the cross derivative $\wn \partial_1 V$ is in $L^2(\psi)$:
\begin{equation}
\label{eq:bound_d12V}
\kappa^2 := 
\int_{\R^{n}} \left| \wn \partial_1 V (x) \right|^2 \, \psi(x) dx < \infty.
\end{equation}
The two above assumptions~\eqref{eq:poincare} and~\eqref{eq:bound_d12V} are very similar to (and actually weaker than) the assumptions [A1] and~[A2] mentioned in the introduction and which have been used in~\cite{legoll-lelievre-10} to study the correctness of the effective dynamics in terms of time marginals.

\bigskip

Third, we assume that the function $b$ defined by~\eqref{eq:def_b} is one-sided Lipschitz on $\R$: there exists $L_b>0$ such that
\begin{equation}
\label{eq:b_lip}
\forall x \in \R, \ \forall y \in \R, \quad
\left( b(y) - b(x) \right) \left( x - y \right)
\leq L_b \left( x - y \right)^2.
\end{equation}
If $b$ is differentiable, this is equivalent to $-b'(x) \leq L_b$ for any $x \in \R$.

In addition, for any $x>0$, we introduce
\begin{equation}
\label{eq:def_alpha_pre}
\alpha(x) = \sup_{s \in [-x,x]} |b'(s)|
\end{equation}
and assume that
\begin{multline}
C_\alpha(\beta) 
= 
\E \left[ \Big( \alpha \left(\left| X^1 \right| \right) \Big)^2 \right] 
=
\E \left[ \left( \sup_{s \in [-|X^1|,|X^1|]} |b'(s)| \right)^2 \right]
\\
=
\int_\R \left( \sup_{s \in [-|\xi|,|\xi|]} |b'(s)| \right)^2 \, \varphi(\xi) \, d\xi
< \infty 
\label{eq:hyp_alpha}
\end{multline}
where $\varphi$ (defined by~\eqref{eq:def_varphi}) is the marginal probability density along the first coordinate $X^1$. The quantity $C_\alpha$ depends on $\beta$ as $\varphi$ and $b$ depend on $\beta$. Note that, if we think of $V$ as having polynomial growth, we see that $\alpha$ also has polynomial growth. In this case, the assumption~\eqref{eq:hyp_alpha} is hence fulfilled. The assumption~\eqref{eq:hyp_alpha} is further discussed in Remarks~\ref{rem:alpha1} and~\ref{rem:alpha2} below. 

We will also sometimes consider the assumption
\begin{equation}
C_{\alpha,p}(\beta) 
= 
\int_\R \left( \sup_{s \in [-|\xi|,|\xi|]} |b'(s)| \right)^{2p/(2-p)} \, \varphi(\xi) \, d\xi
< \infty 
\label{eq:hyp_alpha_p}
\end{equation}
for some $1 \leq p < 2$, which is stronger than~\eqref{eq:hyp_alpha}. Note that~\eqref{eq:hyp_alpha} corresponds to the case $p=1$.

\medskip

Roughly speaking, the assumptions~\eqref{eq:b_lip} and~\eqref{eq:hyp_alpha} will be used below to show that if $x(t)$ and $y(t)$ are solutions to $\dot{x}=-b(x)$ and $\dot{y}=-b(y)+\dot{e}$ (with $x(0)=y(0)$), then $\|x-y\|_{L^\infty(0,T)}$ is small if $\| e \|_{L^\infty(0,T)}$ is small (see Lemma~\ref{lem:gronwall} below).

\begin{remark}
The assumption~\eqref{eq:b_lip} is satisfied if $b$ is Lipschitz on bounded domains and increasing at infinity, which corresponds to a case when the associated free energy $F$ is convex at infinity and smooth.
\end{remark}

\subsection{Main result}

Our main result is the following.

\begin{proposition}
\label{prop:main_global}
Assume that~\eqref{eq:poincare},~\eqref{eq:bound_d12V},~\eqref{eq:b_lip} and~\eqref{eq:hyp_alpha} hold, and that the system starts at equilibrium:
\begin{equation}
\label{eq:equi}
X_0 \sim \mu.
\end{equation}
Consider $(X_t)_{0 \le t \le T}$ solution to~\eqref{eq:X} and $(\xi_t)_{0 \le t \le T}$ solution to~\eqref{eq:dyn_eff} over a bounded time interval $[0,T]$. Then, there exists a constant $C$, that is independent of $\rho$ and $\kappa$, and that only depends on $T$, $C_\alpha(\beta)$ and $L_b$, such that
\begin{equation}
\label{eq:main_result}
\EE \left( \sup_{0 \leq t \leq T} 
\left| X^1_t - \xi_t \right| \right) 
\leq 
C \sqrt{\beta} \ \frac{\kappa}{\rho}.
\end{equation}
\end{proposition}
Note that the constant $C$ in~\eqref{eq:main_result} only depends on $\beta$ through its dependency to $C_\alpha(\beta)$. 

\begin{remark}
\label{rem:holder}
The proof of~\eqref{eq:main_result} also shows that, if we replace the assumption~\eqref{eq:hyp_alpha} by the stronger assumption~\eqref{eq:hyp_alpha_p} for some $1 \leq p < 2$, then there exists a constant $C$, that is independent of $\rho$ and $\kappa$, and that only depends on $p$, $T$, $C_{\alpha,p}(\beta)$ and $L_b$, such that
\begin{equation}
\label{eq:main_result_p}
\EE \left( \sup_{0 \leq t \leq T} 
\left| X^1_t - \xi_t \right|^p \right) 
\leq 
C \left( \sqrt{\beta} \ \frac{\kappa}{\rho} \right)^p.
\end{equation}
See Remark~\ref{rem:holder2} below.
\end{remark}

\begin{remark}
If $\partial_1 V$ is independent of $x_2^n$, then the dynamics~\eqref{eq:X1} is actually closed in $X^1$ and we expect the effective dynamics~\eqref{eq:dyn_eff} to be exact. This is indeed the case: if $\partial_1 V$ is independent of $x_2^n$, then we see from~\eqref{eq:bound_d12V} that $\kappa=0$, and~\eqref{eq:main_result} implies that the effective dynamics is exact.
\end{remark}

Before going into the details in the next sections, let us sketch the proof of Proposition~\ref{prop:main_global}. We introduce 
\begin{equation}
\label{eq:def_f}
f(x) = b(x^1)- \partial_1 V(x)
\end{equation}
and recast~\eqref{eq:X1} in the form
$$
dX^1_t 
=
- b(X^1_t) \, dt + f(X_t) \, dt 
+ \sqrt{2 \beta^{-1}} \, d W^1_t.
$$
Using the definition~\eqref{eq:def_b} of $b$, one gets: for any $\xi \in \R$,
$$
\int_{\R^{n-1}} f(\xi,x_2^n) \ \psi^\xi(x_2^n) \, dx_2^n = 0.
$$
Hence, under some adequate assumptions, for any $\xi$, there exists a unique function $x_2^n \mapsto u(\xi,x_2^n)$ solution to the following Poisson problem:
\begin{equation}
\label{eq:def_u_xi}
L^\xi u = f(\xi,\cdot) 
\quad \text{with} \quad 
\int_{\R^{n-1}} u(\xi,x_2^n) \, \psi^\xi(x_2^n) \, dx_2^n = 0.
\end{equation}
A precise result will be stated below (see Lemma~\ref{lem:poisson}). 
On the one hand, we thus obtain, after integration in time, that
\begin{equation}
\label{eq:X1_exact}
X^1_t - X^1_0 
= - \int_0^t b(X^1_s) \, ds + \sqrt{2 \beta^{-1}} \, W^1_t
+ \int_0^t L^{X^1_s} u(X_s) \, ds.
\end{equation}
On the other hand, integrating the stochastic differential equation~\eqref{eq:dyn_eff} between times $0$ and $t$, we get
\begin{equation}
\label{eq:X1_approx}
\xi_t - X^1_0 = - \int_0^t b(\xi_s) \, ds + \sqrt{2 \beta^{-1}} \, W^1_t.
\end{equation}
We deduce from~\eqref{eq:X1_exact} and~\eqref{eq:X1_approx} that
\begin{equation}
\label{eq:diff}
X^1_t - \xi_t 
= 
\int_0^t \left( b(\xi_s) - b(X^1_s) \right) \, ds + \int_0^t L^{X^1_s} u(X_s) \, ds.
\end{equation}
We then see that we need to bound $\dps \int_0^t L^{X^1_s} u(X_s) \, ds = \int_0^t f(X_s) \,ds$ in the right-hand side of~\eqref{eq:diff} in order to estimate the distance between $X^1_t$ and~$\xi_t$.

The proof is then based on two main arguments. First, we estimate the term $\dps \int_0^t L^{X^1_s} u(X_s) \, ds$ using a result due to T.~Lyons and T.~Zhang in~\cite{lyons-zhang-94} together with an estimate on the solution $u$ to the Poisson problem~\eqref{eq:def_u_xi}, see Section~\ref{sec:LZ}. This estimate relies on the two assumptions~\eqref{eq:poincare} and~\eqref{eq:bound_d12V}.

Second, a Gronwall-type result is proved in Section~\ref{sec:gronwall} to obtain an upper bound on $\dps \EE\left(\sup_{0 \le t \le T} \left| X^1_t - \xi_t \right| \right)$ in terms of $\dps\sqrt{ \EE\left(\sup_{0 \le t \le T} \left| \int_0^t L^{X^1_s} u(X_s) \, ds \right|^2 \right)}$. This result (of interest by its own) relies on the one-sided Lipschitz assumption~\eqref{eq:b_lip} as well as on the integrability assumption~\eqref{eq:hyp_alpha}.

We eventually point out that the stationarity assumption~\eqref{eq:equi} in Proposition~\ref{prop:main_global} can be relaxed using a standard argument based on a conditional expectation with respect to the initial condition, as stated in the following corollary:
\begin{corollary}
\label{coro:main_global}
Assume that~\eqref{eq:poincare},~\eqref{eq:bound_d12V},~\eqref{eq:b_lip} and~\eqref{eq:hyp_alpha} hold. Consider $(X_t)_{0 \le t \le T}$ solution to~\eqref{eq:X} and $(\xi_t)_{0 \le t \le T}$ solution to~\eqref{eq:dyn_eff} over a bounded time interval $[0,T]$, with the initial condition $X_0$ distributed according to a measure $\psi_0(x)\, dx$ such that 
\begin{equation}
\label{eq:borne_moment}
m = \left\| \frac{\psi_0}{\psi} \right\|_{L^\infty(\R^n)} < \infty.
\end{equation}
Then we have 
\begin{equation}
\label{eq:main_result2}
\EE\left( \sup_{0 \leq t \leq T} 
\left| X^1_t - \xi_t \right| \right) 
\leq C \, m \, \sqrt{\beta} \ \frac{\kappa}{\rho},
\end{equation}
where $C$ is the constant of the estimate~\eqref{eq:main_result}.
\end{corollary}
\begin{proof}
Let us introduce 
$$
h(x_0) = \EE^{x_0} \left( \sup_{0 \leq t \leq T} \left| X^1_t - \xi_t \right| \right)
$$
where $\EE^{x_0}$ is the expectation conditionally to the fact that the initial condition of~\eqref{eq:X} is deterministic: $X_0=x_0 \in \R^n$. The result of Proposition~\ref{prop:main_global} is that
$$
\int_{\R^n} h(x) \, \psi(x) \, dx \leq C \sqrt{\beta} \frac{\kappa}{\rho}.
$$
Let us now consider $(X^1_t)_{t \ge 0}$ solution to~\eqref{eq:X} with initial condition distributed according to $\psi_0(x)\, dx$. We have, using~\eqref{eq:borne_moment},
\begin{multline*}
\EE \left( \sup_{0 \leq t \leq T} \left| X^1_t - \xi_t \right| \right)
=
\int_{\R^n} h(x) \, \psi_0(x) \, dx
= 
\int_{\R^n} h(x) \, \psi(x) \, \frac{\psi_0(x)}{\psi(x)} \, dx
\\
\leq 
\left\| \frac{\psi_0}{\psi} \right\|_{L^\infty(\R^n)}
\int_{\R^n} h(x) \, \psi(x) \, dx
\leq 
C \, m \, \sqrt{\beta} \frac{\kappa}{\rho}.
\end{multline*}
This concludes the proof of Corollary~\ref{coro:main_global}. 
\end{proof}

A similar corollary can be stated if we assume~\eqref{eq:hyp_alpha_p} rather than~\eqref{eq:hyp_alpha}, under the assumption that $\dps \left\| \frac{\psi_0}{\psi} \right\|_{L^q(\psi)} < \infty$ for some $q$ (see Remark~\ref{rem:holder}).

\section{An estimate on $f$ and a preliminary result}\label{sec:f}

\subsection{Estimate on $f$}

A direct consequence of the two assumptions~\eqref{eq:poincare} and~\eqref{eq:bound_d12V} is an estimate on the function $f$ defined by~\eqref{eq:def_f}:
$$
f(x) 
= 
b(x^1) - \partial_1 V (x) 
= 
\int_{\R^{n-1}} \partial_1 V(x^1,x_2^n) \, \psi^{x^1}(x_2^n) \, dx_2^n - \partial_1 V (x).
$$
\begin{lemma}\label{lem:f_L2}
Consider $f$ defined by~\eqref{eq:def_f} and assume that the conditional probability measures $\psi^{\xi}(x_2^n) \, dx_2^n$ satisfy the Poincaré inequalities~\eqref{eq:poincare}, and that the cross derivative $\partial_1 \wn V$ satisfies~\eqref{eq:bound_d12V}. Then we have
\begin{equation}
\label{eq:f_L2}
\int_{\R^n} f^2 \, \psi \le \frac{\kappa^2}{\rho}.
\end{equation}
\end{lemma}

\begin{proof}
Notice that for any $\xi \in \R$, $\dps \int_{\R^{n-1}} f(\xi,x_2^n) \, \psi^\xi(x_2^n) \, dx_2^n=0$. Thus, using~\eqref{eq:poincare}, we get, for any fixed $\xi \in \R$,
$$
\int_{\R^{n-1}} (f(\xi,x_2^n))^2 \, \psi^\xi(x_2^n) \, dx_2^n
\le 
\rho^{-1} \int_{\R^{n-1}} \left| \wn \partial_1 V (\xi,x_2^n) \right|^2 \, \psi^\xi(x_2^n) \, dx_2^n.
$$
By multiplying by $\varphi(\xi)$, integrating over $\xi \in \R$ and using~\eqref{eq:bound_d12V}, we obtain~\eqref{eq:f_L2}.
\end{proof}

Notice that, as a corollary of~\eqref{eq:f_L2}, since
$$
\int_{\R} \left[ \int_{\R^{n-1}} \left( f(\xi,x_2^n) \right)^2 \, \psi^\xi(x_2^n) \, dx_2^n \right] \varphi(\xi) \, d\xi
=
\int_{\R^n} f^2 \, \psi < \infty
$$
and $\varphi > 0$, we have that, for almost all $\xi \in \R$,
\begin{equation}
\label{eq:f_L2_xi}
\int_{\R^{n-1}} f^2(\xi,x_2^n) \, \psi^\xi(x_2^n) \, dx_2^n < \infty.
\end{equation}

\subsection{A simple consequence of Lemma~\ref{lem:f_L2}}\label{sec:simple_csq}

Before proving our main result Proposition~\ref{prop:main_global}, we first state a preliminary result, which is weaker but also much more simple to prove than Proposition~\ref{prop:main_global}. This result already highlights the importance of the assumptions~\eqref{eq:poincare},~\eqref{eq:bound_d12V} and~\eqref{eq:b_lip}.

\begin{lemma}
\label{lem:main_global}
Assume that~\eqref{eq:poincare},~\eqref{eq:bound_d12V},~\eqref{eq:b_lip} and~\eqref{eq:equi} hold. Consider $(X_t)_{0 \le t \le T}$ solution to~\eqref{eq:X} and $(\xi_t)_{0 \le t \le T}$ solution to~\eqref{eq:dyn_eff} over a bounded time interval $[0,T]$. Then, there exists a constant $C$, that depends on $T$ and $L_b$, but is independent of $\beta$, $\rho$ and $\kappa$, such that
\begin{equation}
\label{eq:main_result_bis_global}
\EE \left( \sup_{0 \leq t \leq T} 
\left( X^1_t - \xi_t \right)^2 \right) 
\leq 
C \frac{\kappa^2}{\rho}.
\end{equation}
\end{lemma}
Note that we do not need the assumption~\eqref{eq:hyp_alpha} here. The above result obviously implies that
$$
\EE \left( \sup_{0 \leq t \leq T} \left| X^1_t - \xi_t \right| \right) \leq C \frac{\kappa}{\sqrt{\rho}},
$$
a result which is weaker than Proposition~\ref{prop:main_global} as we think of $\rho$ as being large. The stationarity assumption~\eqref{eq:equi} can be weakened in a similar way as in Corollary~\ref{coro:main_global} above by using the same conditioning argument.

\begin{proof}
From~\eqref{eq:X1} and~\eqref{eq:dyn_eff}, and using the definition~\eqref{eq:def_f} of $f$, we have
$$
d(X^1_t - \xi_t) 
=
\left( b(\xi_t) - b(X^1_t) \right) \, dt + f(X_t) \, dt.
$$
We deduce from an It\^o's computation that 
\begin{align*}
\frac12 \left( X^1_t - \xi_t \right)^2 
&=
\int_0^t \left( X^1_s - \xi_s \right) 
\left( b(\xi_s) - b(X^1_s) \right) \, ds 
+
\int_0^t \left( X^1_s - \xi_s \right) f(X_s) \, ds
\\
&\leq 
\left(L_b +\frac12\right) \int_0^t \left( \xi_s - X^1_s \right)^2 ds
+
\frac12 \int_0^t \left( f(X_s) \right)^2 \, ds,
\end{align*}
where we have used~\eqref{eq:b_lip} and a Young inequality in the last line. Setting $\dps \phi(t) = \frac12 \int_0^t \left( X^1_s - \xi_s \right)^2 ds$ and $\dps M = \frac12 \int_0^T \left( f(X_s) \right)^2 \, ds$, we thus see that
\begin{equation}
\label{eq:tuut_}
\phi'(t) 
\leq M + (2 L_b +1) \phi(t).
\end{equation}
Using Gronwall lemma and the fact that $\phi(0)= 0$, we deduce that, for any $t \in [0,T]$,
$$
\phi(t) \leq \frac{M}{2L_b+1} \left( e^{(2L_b+1) t} - 1 \right).
$$
Hence, using~\eqref{eq:tuut_}, we obtain
\begin{align*}
\frac12 \left( X^1_t - \xi_t \right)^2
=
\phi'(t)
\leq
e^{(2 L_b+1) t} M
\le\frac12 e^{(2 L_b+1) T} \ \int_0^T \left( f(X_s) \right)^2 \, ds,
\end{align*}
where the right-hand side is independent of $t$. Taking the supremum over $t\in [0,T]$ and taking expectations, we deduce that
$$
\EE \left( \sup_{0 \leq t \leq T} \left( X^1_t - \xi_t \right)^2 \right) 
\leq 
e^{(2 L_b+1) T} \ \EE \left[ \int_0^T \left( f(X_s) \right)^2 \, ds \right].
$$
Now, one can use Lemma~\ref{lem:f_L2} above to control the right-hand side. Indeed, since $X_0$ (and thus $X_t$ at any time $t$) is distributed according to the equilibrium measure (see~\eqref{eq:equi}), we obtain, using~\eqref{eq:f_L2}, that
\begin{align*}
\EE \left( \sup_{0 \leq t \leq T} \left( X^1_t - \xi_t \right)^2 \right) 
\leq 
e^{(2 L_b+1) T} \ T \int_{\R^n} f^2 \psi
\leq
T e^{(2 L_b+1) T} \ \frac{\kappa^2}{\rho}. 
\end{align*}
This proves the claimed bound~\eqref{eq:main_result_bis_global} and concludes the proof of Lemma~\ref{lem:main_global}.
\end{proof}

\section{Estimate on the term $\dps \int_0^t L^{X^1_s} u(X_s) \, ds$ in~\eqref{eq:diff}}
\label{sec:LZ}

The aim of this section is to get an estimate in terms of $\kappa$ and $\rho$ on the last term $\dps \int_0^t L^{X^1_s} u(X_s) \, ds= \int_0^t f(X_s) \, ds$ in~\eqref{eq:diff} (where, we recall, $u$ is the solution to the Poisson problem~\eqref{eq:def_u_xi}) using the estimate on $f$ of Lemma~\ref{lem:f_L2}, and assuming that $X_0$ is distributed according to the equilibrium measure~$\mu$. As explained in Section~\ref{sec:landim_olla}, it is enough to estimate $\dps \int_{\R^n} \left| \wn u \right|^2 \psi$. The well-posedness of~\eqref{eq:def_u_xi} and a bound on $\dps \int_{\R^n} \left| \wn u \right|^2 \psi$ is shown in Section~\ref{sec:poisson}. Finally, Section~\ref{sec:estim_last_term} collects the results of Sections~\ref{sec:poisson} and~\ref{sec:landim_olla} to get an estimate on $\dps \int_0^t L^{X^1_s} u(X_s) \, ds$.

In all this section, only the two assumptions~\eqref{eq:poincare} and~\eqref{eq:bound_d12V} are needed.

\subsection{The Poisson problem}\label{sec:poisson}

Let us first state an existence and uniqueness result for the Poisson problem~\eqref{eq:def_u_xi} introduced above, as well as an estimate on its solution.
\begin{lemma}
\label{lem:poisson}
Assume that~\eqref{eq:poincare} and~\eqref{eq:bound_d12V} hold, and consider the function $f$ defined by~\eqref{eq:def_f}. Then, for any $\xi \in \R$, there exists a unique function $x_2^n \mapsto u(\xi,x_2^n)$ in $H^1(\psi^\xi)$ such that
\begin{equation}
\label{eq:def_u_xi_prime}
L^\xi u = f(\xi,\cdot) 
\quad \text{with} \quad 
\int_{\R^{n-1}} u(\xi,x_2^n) \, \psi^\xi(x_2^n) \, dx_2^n = 0,
\end{equation}
where the functional space $H^1(\psi^\xi)$ is defined by~\eqref{eq:H1}. Moreover, $u$ is a ${\mathcal C}^\infty$ function and satisfies the estimate
\begin{equation}
\label{eq:bound_der_u}
\int_{\R^{n-1}} \left| \wn u(\xi,x_2^n) \right|^2 
\, \psi^\xi(x_2^n) \, dx_2^n
\leq 
\frac{\beta^2}{\rho} \int_{\R^{n-1}} f^2(\xi,x_2^n) 
\, \psi^\xi(x_2^n) \, dx_2^n.
\end{equation} 
In addition, we have
\begin{equation}
\label{eq:estim_wnu}
\int_{\R^n} \left| \wn u \right|^2 \psi \le \beta^2 \, \frac{\kappa^2}{\rho^2}.
\end{equation}
\end{lemma}

\begin{proof}
Let us introduce the functional space 
$$
H^1_m(\psi^\xi)
=
\left\{v \in H^1(\psi^\xi), \quad \int_{\R^{n-1}} v(x_2^n) \, \psi^\xi(x_2^n) \, dx_2^n = 0 \right\}.
$$
A variational formulation of~\eqref{eq:def_u_xi_prime} is the following: find $u(\xi,\cdot) \in H^1_m(\psi^\xi)$ such that, for all $v \in H^1_m(\psi^\xi)$,
\begin{equation}
\label{eq:def_u_xi_var}
-\beta^{-1} \int_{\R^{n-1}} \wn u(\xi,\cdot) \cdot \wn v \ \psi^\xi = \int_{\R^{n-1}} f(\xi,\cdot) \ v \ \psi^\xi.
\end{equation}
Here, we used the fact that for any smooth functions $u : \R^n \to \R$ and $v : \R^{n-1} \to \R$,
$$
\int_{\R^{n-1}} (L^\xi u) \, v \, \psi^\xi = - \beta^{-1} \int_{\R^{n-1}} \wn u(\xi,\cdot) \cdot \wn v \ \psi^\xi.
$$
The variational problem~\eqref{eq:def_u_xi_var} admits a unique solution using Lax-Milgram lemma and~\eqref{eq:poincare} to get the coercivity of the bilinear form. The right-hand-side in~\eqref{eq:def_u_xi_var} is well defined in view of~\eqref{eq:f_L2_xi}. Moreover, the solution to~\eqref{eq:def_u_xi_var} is indeed a solution to~\eqref{eq:def_u_xi_prime} (in distributional sense, say) since $\dps \int_{\R^{n-1}} f \, \psi^\xi=0$.

By standard elliptic regularity results, since the functions $V$ and $f$ are assumed to be smooth, the function $u$ is actually ${\mathcal C}^\infty$. 

By taking $v=u$ in~\eqref{eq:def_u_xi_var}, we get
\begin{align*}
\int_{\R^{n-1}} \left| \wn u \right|^2 \, \psi^\xi 
&= -\beta \int_{\R^{n-1}} f \, u \, \psi^\xi
\\
&\le \beta \left( \int_{\R^{n-1}} f^2 \, \psi^\xi\right)^{1/2} \left( \int_{\R^{n-1}} u^2 \, \psi^\xi\right)^{1/2}
\\
&\le \frac{\beta}{\sqrt{\rho}} \left( \int_{\R^{n-1}} f^2 \, \psi^\xi\right)^{1/2} \left( \int_{\R^{n-1}} \left| \wn u \right|^2 \psi^\xi\right)^{1/2}
\end{align*}
where we used~\eqref{eq:poincare} in the last line. This yields~\eqref{eq:bound_der_u}. By combining~\eqref{eq:f_L2} and~\eqref{eq:bound_der_u}, we get~\eqref{eq:estim_wnu}. This concludes the proof.
\end{proof}

\subsection{An estimate on square-integrable martingales} 
\label{sec:landim_olla}

The following general result (see~\cite{lyons-zhang-94} or~\cite[Section 2.5, Lemma 2.4]{komorowski-landim-olla-12}) is useful for our proof.

\begin{lemma}
\label{lem:landim_olla}
Let $(X_t)_{t \geq 0}$ be the solution to~\eqref{eq:X}, with its initial condition distributed according to the equilibrium measure $\mu$ (see assumption~\eqref{eq:equi}). Consider a function $\Phi: \R^n \to \R^n$ such that $\Phi \in ({\mathcal C}^\infty \cap L^2(\psi))^n$. Then, for any $T$, we have
\begin{equation}
\label{eq:resu:olla}
\EE \left[ \sup_{0 \leq t \leq T} \left| \int_0^t \nabla^\star \Phi(X_s) \, ds \right|^2 \right] 
\leq
8 T \beta \, \| \Phi \|^2_{L^2(\psi)},
\end{equation}
where, we recall, $\nabla^\star$ denotes the adjoint of the operator $\nabla$ with respect to the $L^2(\psi)$ scalar product, so that
$$ 
\nabla^\star \Phi = \beta \nabla V \cdot \Phi - {\rm div}(\Phi).
$$
\end{lemma}

We have not looked at the minimal regularity assumptions on $\Phi$ for this lemma to hold, since we use it below for a function $\Phi$ which is indeed in $({\mathcal C}^\infty \cap L^2(\psi))^n$. We refer to~\cite[Section 2.5, Lemma 2.4]{komorowski-landim-olla-12} for a statement under weaker regularity assumptions on $\Phi$.

\begin{proof} The proof falls in two steps. 

\medskip

\noindent
{\bf Step 1.} For any $\eta>0$, consider the resolvent problem
\begin{equation}
\label{eq:poisson_L_prime_eta}
\eta w_\eta - L w_\eta = - \nabla^\star \Phi,
\end{equation}
a variational formulation of which is: find $w_\eta \in H^1(\psi)$ such that, for any test function $v \in H^1(\psi)$,
$$
\eta \int_{\R^n} w_\eta \, v \, \psi + \beta^{-1} \int_{\R^n} \nabla w_\eta \cdot \nabla v \, \psi = - \int_{\R^n} \Phi \cdot \nabla v \, \psi.
$$
Using the Lax-Milgram theorem, the above problem has a unique solution $w_\eta \in H^1(\psi)$. Furthermore, taking $v \equiv w_\eta$ as function test in the above variational formulation, we get
$$
\eta \| w_\eta \|^2_{L^2(\psi)} + \beta^{-1} \| \nabla w_\eta \|^2_{L^2(\psi)} \leq \| \Phi \|_{L^2(\psi)} \ \| \nabla w_\eta \|_{L^2(\psi)},
$$
which hence shows that, for any $\eta>0$, we have
\begin{equation}
\label{eq:a_priori}
\| \nabla w_\eta \|_{L^2(\psi)} \leq \beta \ \| \Phi \|_{L^2(\psi)}.
\end{equation}
Furthermore, we have $\dps \sqrt{\eta} \ \| w_\eta \|_{L^2(\psi)} \leq \sqrt{\beta} \ \| \Phi \|_{L^2(\psi)}$, hence
\begin{equation}
\label{eq:a_priori2}
\lim_{\eta \to 0} \eta \ \| w_\eta \|_{L^2(\psi)} = 0.
\end{equation}
In addition, by standard elliptic regularity results, since the functions $V$ and $\Phi$ are assumed to be smooth, the function $w_\eta$ is actually ${\mathcal C}^\infty$. 

\medskip

\noindent
{\bf Step 2.} Now, let us consider a fixed time $T>0$. Since $w_\eta$ is smooth, we can write, by It\^o's calculus, that, for any $t \in [0,T]$,
\begin{equation}
\label{eq:alpha1_}
\begin{split}
w_\eta(X_t) - w_\eta(X_0) = \int_0^t Lw_\eta(X_s) \, ds + \sqrt{2 \beta^{-1}}\int_0^t \nabla w_\eta(X_s) \cdot dW_s.
\end{split}
\end{equation}
Let us introduce, for $s \in [0,T]$, 
$$
Y_s = X_{T-s}.
$$
Since $X_0$ is distributed according to the equilibrium measure~$\mu$ and $(X_t)_{t \ge 0}$ is reversible with respect to~$\mu$, $(Y_s)_{0 \le t \le T}$ has the same law as $(X_s)_{0 \le t \le T}$. We thus can write
$$
dY_s = - \nabla V(Y_s) \, ds + \sqrt{2 \beta^{-1}} \, d \overline{W}_s
$$
with $Y_0 = X_T$ and where $(\overline{W}_s)_{0 \le s\le T}$ is a Brownian motion. Similarly to~\eqref{eq:alpha1_}, we have, for any $t \in [0,T]$,
\begin{equation}
\label{eq:alpha2_}
w_\eta(Y_T) - w_\eta(Y_{T-t}) = \int_{T-t}^T Lw_\eta(Y_s) \, ds + \sqrt{2 \beta^{-1}}\int_{T-t}^T \nabla w_\eta(Y_s) \cdot d\overline{W}_s.
\end{equation}
Setting
$$
M_t = \int_0^t \nabla w_\eta(X_s) \cdot dW_s 
\ \text{ and } \
\overline{M}_t = \int_{T-t}^T \nabla w_\eta(Y_s) \cdot d\overline{W}_s,
$$
we deduce from adding~\eqref{eq:alpha1_} and~\eqref{eq:alpha2_} that
\begin{align*}
0
&= 
\int_0^t Lw_\eta(X_s) \, ds + \int_{T-t}^T Lw_\eta(Y_s) \, ds 
+ \sqrt{2 \beta^{-1}} \left( M_t + \overline{M}_t \right)
\\
&=
\int_0^t Lw_\eta(X_s) \, ds + \int_0^t Lw_\eta(Y_{T-s}) \, ds 
+ \sqrt{2 \beta^{-1}} \left( M_t + \overline{M}_t \right)
\\
&=
2 \int_0^t Lw_\eta(X_s) \, ds + 
\sqrt{2 \beta^{-1}} \left( M_t + \overline{M}_t \right).
\end{align*}
Hence
$$
4 \EE \left[ \sup_{0 \leq t \leq T} 
\left| \int_0^t Lw_\eta(X_s) \, ds \right|^2 \right] 
\leq
4 \beta^{-1} \left( 
\EE \left[ \sup_{0 \leq t \leq T} \left| M_t \right|^2 \right]
+
\EE \left[ \sup_{0 \leq t \leq T} \left| \overline{M}_t \right|^2 \right]
\right).
$$
The random process $(M_t)_{0 \le t \le T}$ is a martingale since $\dps \E\int_0^T |\nabla w_\eta|^2 (X_s) \, ds = T \int_{\R^n} |\nabla w_\eta|^2 \, \psi < \infty$ in view of~\eqref{eq:a_priori} and the fact that, for any time $s \in [0,T]$, $X_s$ is distributed according to $\mu$. Using Doob inequality on the martingales $(M_t)_{0 \le t \le T}$ and $(\overline{M}_t)_{0 \le t \le T}$, which reads
$$
\EE \left[ \sup_{0 \leq t \leq T} \left| M_t \right|^2 \right]
\leq 
4 \EE \left[ \left| M_T \right|^2 \right]
$$
and likewise for $\overline{M}_t$, we obtain
\begin{align}
\EE \left[ \sup_{0 \leq t \leq T} 
\left| \int_0^t Lw_\eta(X_s) \, ds \right|^2 \right] 
& \leq
8 \beta^{-1} \EE \left[ \int_0^T \left| \nabla w_\eta(X_s) \right|^2 ds \right]
\nonumber \\
&=
8 \beta^{-1} T \int_{\R^n} \left| \nabla w_\eta \right|^2 \psi
\nonumber \\
& \leq
8 \beta T \| \Phi \|^2_{L^2(\psi)},
\label{eq:olla2}
\end{align}
where we have used~\eqref{eq:a_priori} in the last line.

\medskip

In view of~\eqref{eq:poisson_L_prime_eta}, we now write, for any $\nu > 0$, that
$$
\left| \int_0^t \nabla^\star \Phi(X_s) \, ds \right|^2 \leq (1+\nu) \left| \int_0^t L w_\eta(X_s) \, ds \right|^2 + \left( 1+ \frac{1}{\nu}\right) \left| \int_0^t \eta w_\eta(X_s) \, ds \right|^2,
$$
thus
\begin{eqnarray*}
&& \EE \left[ \sup_{0 \leq t \leq T} \left| \int_0^t \nabla^\star \Phi(X_s) \, ds \right|^2 \right]
\\
& \leq & 
(1+\nu) \EE \left[ \sup_{0 \leq t \leq T} \left| \int_0^t L w_\eta(X_s) \, ds \right|^2 \right] 
+
\left( 1+ \frac{1}{\nu}\right) \EE \left[ \sup_{0 \leq t \leq T} \left| \int_0^t \eta w_\eta(X_s) \, ds \right|^2 \right]
\\
& \leq & 
(1+\nu) \, 8 \beta T \| \Phi \|^2_{L^2(\psi)} + \left( 1+ \frac{1}{\nu}\right) \eta^2 T \, \EE \left[ \int_0^T w^2_\eta(X_s) \, ds \right]
\\
& \leq & 
(1+\nu) \, 8 \beta T \| \Phi \|^2_{L^2(\psi)} + \left( 1+ \frac{1}{\nu}\right) \eta^2 T^2 \, \|w_\eta\|^2_{L^2(\psi)}.
\end{eqnarray*}
We now pass to the limit $\eta \to 0$ using~\eqref{eq:a_priori2}. We get that, for any $\nu > 0$,
$$
\EE \left[ \sup_{0 \leq t \leq T} \left| \int_0^t \nabla^\star \Phi(X_s) \, ds \right|^2 \right]
\leq
(1+\nu) \, 8 \beta T \| \Phi \|^2_{L^2(\psi)},
$$
which implies~\eqref{eq:resu:olla}. This concludes the proof of Lemma~\ref{lem:landim_olla}.
\end{proof}

\subsection{Estimate on the term $\dps \int_0^t L^{X^1_s} u(X_s) \, ds$}
\label{sec:estim_last_term}

We are now in position to bound the last term in~\eqref{eq:diff}. 

\begin{proposition}
\label{prop:bb}
Let $X_t$ be the solution to~\eqref{eq:X}, with its initial condition distributed according to the equilibrium measure $\mu$ (see assumption~\eqref{eq:equi}). We assume that~\eqref{eq:poincare} and~\eqref{eq:bound_d12V} hold. Let $u$ be defined as the solution to the Poisson equation~\eqref{eq:def_u_xi_prime}. Then, we have
\begin{equation}
\label{eq:bb}
\EE \left[ \sup_{0 \leq t \leq T} \left| \int_0^t L^{X^1_s} u(X_s) \, ds\right|^2 \right] \leq 8 T \beta \frac{\kappa^2}{\rho^2}.
\end{equation}
\end{proposition}

\begin{proof}
Let us introduce $\Phi=(0,\wn u):\R^n \to \R^n$. We have
\begin{align*}
L^\xi u 
&= - \wn V \cdot \wn u + \beta^{-1} \widehat{\Delta} u \\
&= \beta^{-1} ( -\beta \wn V \cdot \wn u + \widehat{\Delta} u) \\
&= - \beta^{-1} \nabla^\star \Phi.
\end{align*}
Since $u \in {\mathcal C}^\infty$, we have $\Phi \in ({\mathcal C}^\infty)^n$. Furthermore, in view of~\eqref{eq:estim_wnu}, we have $\Phi \in (L^2(\psi))^n$. We are thus in position to use~\eqref{eq:resu:olla} and get that
$$
\EE \left[ \sup_{0 \leq t \leq T} \left| \int_0^t L^{X^1_s} u(X_s) \, ds\right|^2 \right] \leq 8 T \beta^{-1} \int_{\R^n} \left| \wn u \right|^2 \psi.
$$
The estimate~\eqref{eq:bb} is then obtained as a direct consequence of the estimate~\eqref{eq:estim_wnu} on the solution $u$ to the Poisson problem.
\end{proof}


\section{A Gronwall-type result}
\label{sec:gronwall}

In this section, we state a general Gronwall-type result that will be crucial to prove Proposition~\ref{prop:main_global}. This section explains the role of the one-sided Lipschitz assumption~\eqref{eq:b_lip} as well as the integrability condition~\eqref{eq:hyp_alpha}.

\begin{lemma}
\label{lem:gronwall}
Consider a smooth function $b: \R \to \R$ satisfying the one-sided Lipschitz assumption~\eqref{eq:b_lip} with constant $L_b$ and the integrability condition~\eqref{eq:hyp_alpha} with constant $C_\alpha$. Let $X_t \in \R$ and $Y_t \in \R$ be the solutions to
\begin{align}
\label{eq:un}
d X_t &= -b(X_t) dt + \sigma dB_t \,,
\\
\label{eq:deux}
d Y_t &= -b(Y_t) dt + \sigma dB_t + f_t dt \, ,
\end{align}
for some time integrable stochastic process $f_t$ and some positive constant $\sigma$, where $B_t$ is a one-dimensional Brownian motion. We assume that there exists a probability measure $\varphi(x) \, dx$ on $\R$ which is invariant both for the dynamics~\eqref{eq:un} and~\eqref{eq:deux} and that $X_0=Y_0$ are distributed according to that measure.

Consider a fixed time interval $[0,t]$. We have 
\begin{equation}
\label{eq:resu}
\E \left[ \sup_{s \in [0,t]} |X_s-Y_s| \right]
\leq
C \, \sqrt{\EE \left[ \sup_{s \in [0,t]} |e_s|^2 \right]}
\end{equation}
where $\displaystyle e_s=\int_0^s f_\tau \, d\tau$ and $C$ is a constant only depending on $t$, $C_\alpha$ and $L_b$.
\end{lemma}


Let us comment on this result. First, using the one-sided assumption~\eqref{eq:b_lip}, one can check that $d |X_t-Y_t| \le \left( L_b |X_t - Y_t| \, + |f_t| \right) \, dt$, which yields, by the Gronwall lemma, an estimate of the difference $(|X_s-Y_s|)_{s \in [0,t]}$ in terms of $(|f_s|)_{s \in [0,t]}$. Here, we only assume a control on $\dps \left(\left|\int_0^s f_{\tau} d\tau\right|\right)_{s \in [0,t]}$. This is why the standard approach does not apply, and why we need Assumption~\eqref{eq:hyp_alpha} in addition to the one-sided Lipschitz assumption~\eqref{eq:b_lip}. 

Second, if $b$ is assumed to be globally Lipschitz with constant $L_b$ (instead of one-sided Lipschitz), one can check that $\dps |X_t-Y_t| \le L_b \int_0^t |X_s - Y_s| \, ds + |e_t|$ and, again by Gronwall lemma, $\dps \sup_{s \in [0,t]} |X_s-Y_s| \le e^{L_b t} \sup_{s \in [0,t]} |e_s|$. Here, we do not assume $b$ Lipschitz but only one-sided Lipschitz, and we were not able to prove such a pathwise inequality under the only assumption~\eqref{eq:b_lip}. Actually, we can prove a similar inequality but only in expectation, see~\eqref{eq:resu}, and under the additional assumption~\eqref{eq:hyp_alpha}.

\medskip

Concerning the dependency of the constant $C$ in~\eqref{eq:resu} on the time $t$, one can check from the proof that $C=1+\widetilde{C} t \exp(L_b t)$ where $\widetilde{C}$ is a constant only depending on $C_\alpha$ and $L_b$. Notice that if $b$ is assumed to be increasing, one can take $L_b=0$, so that $C=1+\widetilde{C} t$.

\medskip

Before proving Lemma~\ref{lem:gronwall}, let us comment on Assumption~\eqref{eq:hyp_alpha}, that reads, we recall, 
$$ 
C_\alpha(\beta) 
= 
\int_\R \left( \sup_{s \in [-|x|,|x|]} |b'(s)| \right)^2 \, \varphi(x) \, dx
< \infty. 
$$

\begin{remark}
\label{rem:alpha1}
It is not enough to assume 
\begin{equation}
\label{eq:integ2}
\int_\R |b'(x)|^2 \, \varphi(x) \, dx < \infty
\end{equation}
for~\eqref{eq:hyp_alpha} to hold. Consider indeed the case of
$
\varphi(x) = (1+|x|^2)^{-1}
$
and $b'=0$ on $\R$, except in $[n,n+1/n^2]$, where $|b'|=n$, for any $n \in \N^\star$. Then
$$
\int_\R |b'(x)|^2 \, \varphi(x) \, dx
\leq
\sum_{n \geq 1} n^2 (1+|n|^2)^{-1} \frac{1}{n^2} < \infty,
$$
and~\eqref{eq:integ2} holds. But, for any $x$ such that $n<|x| < n+1$, $\dps \sup_{s \in [-|x|,|x|]} |b'(s)| = n$, thus
$$
\int_\R \left( \sup_{s \in [-|x|,|x|]} |b'(s)| \right)^2 \, \varphi(x) \, dx
\geq
2 \sum_{n \geq 0} n^2 \frac{1}{1+|n+1|^2} = \infty
$$
and~\eqref{eq:hyp_alpha} does not hold.
\end{remark}

\begin{remark}
\label{rem:alpha2}
Rather than assuming~\eqref{eq:hyp_alpha}, another possibility is to assume that $b'$ satisfies~\eqref{eq:integ2} and that there exists $C$ such that, for any $x$ and $y$,
$$
\sup_{\theta \in (x,y)} \left| b'(\theta) \right| \leq C \Big( \left| b'(x) \right| + \left| b'(y) \right| \Big).
$$
Then~\eqref{eq:resu} again holds.
\end{remark}

\begin{proof}[Proof of Lemma~\ref{lem:gronwall}]
As an obvious consequence of the assumptions, we see that, at any time $t$, $X_t$ and $Y_t$ share the same probability law $\varphi(x) \, dx$, independent of $t$. Let $Z_t = Y_t - e_t$. We infer from~\eqref{eq:un} and~\eqref{eq:deux} that
\begin{align}
& d |X_t - Z_t|
\nonumber
\\
&=
( b(Z_t + e_t) - b(X_t) ) \, \sgn(X_t-Z_t) dt
\nonumber
\\
&=
\left( 1_{|X_t - Z_t| < |e_t|} + 1_{|X_t - Z_t| \geq |e_t|} \right) ( b(Z_t + e_t) - b(X_t) ) \, \sgn(X_t-Z_t) dt.
\label{eq:sor0}
\end{align}
To control the second term in the above right-hand side (corresponding to the case $|X_t - Z_t| \geq |e_t|$), we argue as follows:
\begin{itemize}
\item If $X_t \geq Z_t$, then it means that $X_t - Z_t \geq |e_t| \geq e_t$, hence $X_t \geq Z_t + e_t$, hence, using~\eqref{eq:b_lip}, we have $b(Z_t+e_t) - b(X_t) \leq L_b(X_t - Z_t- e_t)$. Therefore,
\begin{align}
&
1_{X_t \geq Z_t} 1_{|X_t - Z_t| \geq |e_t|} ( b(Z_t + e_t) - b(X_t)) \, \sgn(X_t-Z_t) 
\nonumber
\\
&=
( b(Z_t + e_t) - b(X_t) ) 1_{X_t - Z_t \geq |e_t|}
\nonumber
\\
& \leq
L_b(X_t - Z_t- e_t) 1_{X_t - Z_t \geq |e_t|}
\nonumber
\\
& \leq
L_b (|X_t - Z_t| + |e_t|).
\label{eq:sor1}
\end{align}
\item If $X_t \leq Z_t$, then it means that $X_t - Z_t \leq - |e_t| \leq e_t$, hence $X_t \leq Z_t + e_t$, hence, using~\eqref{eq:b_lip}, we get $b(Z_t+e_t) - b(X_t) \geq L_b (X_t- Z_t-e_t)$. Therefore,
\begin{align}
&1_{X_t \leq Z_t} 1_{|X_t - Z_t| \geq |e_t|} ( b(Z_t + e_t) - b(X_t) ) \, \sgn(X_t-Z_t) 
\nonumber
\\
&=- ( b(Z_t + e_t) - b(X_t)) 1_{X_t - Z_t \leq - |e_t|}
\nonumber
\\
& \leq - L_b(X_t - Z_t- e_t) 1_{X_t - Z_t \leq - |e_t|}
\nonumber
\\
&\leq L_b (|X_t - Z_t| + |e_t|).
\label{eq:sor2}
\end{align}
\end{itemize}
Collecting~\eqref{eq:sor0}, \eqref{eq:sor1} and~\eqref{eq:sor2}, we have
\begin{equation}
\label{eq:titi1}
d |X_t - Z_t|
\le
1_{|X_t - Z_t|< |e_t|} | b(X_t) - b(Z_t + e_t) | dt + L_b (|X_t - Z_t| + |e_t|) dt.
\end{equation}
To proceed, we write
$$
\left| b(X_t) - b(Z_t + e_t) \right| 
=
\left| b(X_t) - b(Y_t) \right| 
= 
\left| X_t - Y_t \right| \, \left| b'(\theta_t) \right|
$$
for some $\theta_t \in (X_t,Y_t)$. Using the function $\alpha$ defined by~\eqref{eq:def_alpha_pre}, we obtain
$$
\left| b(X_t) - b(Z_t + e_t) \right| 
\leq 
\left| X_t - Y_t \right| \, \left( \alpha(|X_t|) + \alpha(|Y_t|) \right).
$$
Inserting this relation in~\eqref{eq:titi1}, we obtain
\begin{align*}
&d |X_t - Z_t|
\\
&\le
1_{|X_t - Z_t|< |e_t|} \left| X_t - Y_t \right| \, \left( \alpha(|X_t|) + \alpha(|Y_t|) \right) \, dt + L_b (|X_t - Z_t| + |e_t|) dt
\\
&=
1_{|X_t - Z_t|< |e_t|} \left| X_t - Z_t - e_t \right| \, \left( \alpha(|X_t|) + \alpha(|Y_t|) \right) \, dt + L_b (|X_t - Z_t| + |e_t|) dt
\\
&\le
2 \left( \alpha(|X_t|) + \alpha(|Y_t|) + L_b/2 \right) \, |e_t| \, dt + L_b |X_t - Z_t| dt.
\end{align*}
We thus obtain
$$
d \Big[ \exp(-L_bt) |X_t - Z_t| \Big]
\le
2 \exp(-L_bt) \left( \alpha(|X_t|) + \alpha(|Y_t|) + L_b/2 \right) \, |e_t| \, dt.
$$
We now integrate in time, and use the fact that $X_0 = Z_0$:
\begin{align*}
\exp(-L_bt) |X_t-Z_t|
&\le
2 \int_0^t \exp(-L_bs) \left( \alpha(|X_s|) + \alpha(|Y_s|) +L_b/2 \right) \, |e_s| \, ds
\\
&\le
2 \left( \sup_{ s \in [0,t] } |e_s| \right) \int_0^t \left( \alpha(|X_s|) + \alpha(|Y_s|) + L_b/2 \right) \, ds.
\end{align*}
We deduce that
\begin{equation}
\label{eq:holder}
\sup_{s \in [0,t]} |X_s-Z_s|
\le 
2 \exp(L_bt) \left( \sup_{ s \in [0,t] } |e_s| \right) \int_0^t \left( \alpha(|X_s|) + \alpha(|Y_s|) + L_b/2 \right) \, ds,
\end{equation}
thus
\begin{align*}
&
\E \left[ \sup_{s \in [0,t]} |X_s-Z_s| \right]
\\
&\le
2 \exp(L_bt) \sqrt{\E \left[ \sup_{ s \in [0,t] } |e_s|^2 \right]} \ \sqrt{ \E \left[ \int_0^t \left( \alpha(|X_s|) + \alpha(|Y_s|) + L_b/2\right) \, ds \right]^2 }
\\
&\le
2 \exp(L_bt) \eta \sqrt{t} \ \sqrt{ \E \left[ \int_0^t \Big( \alpha(|X_s|) + \alpha(|Y_s|) + L_b/2\Big)^2 \, ds \right] }
\\
&\le
2 \exp(L_bt) \eta \sqrt{3t} \ \sqrt{ \E \left[ \int_0^t \Big[ \Big( \alpha(|X_s|) \Big)^2 + \Big( \alpha(|Y_s|) \Big)^2 + \frac{L_b^2}{4} \Big] \, ds \right]}
\end{align*}
where $\eta=\sqrt{\E \left[ \sup_{ s \in [0,t] } |e_s|^2 \right]}$. Since the law of $X_s$ and $Y_s$ is the same and does not depend on time, $\dps \E\left[ \alpha(|X_s|)^2\right]=\E\left[ \alpha(|Y_s|)^2\right]=\E\left[ \alpha(|X_0|)^2\right]$ so that
%
\begin{equation}
\label{eq:titi2}
\E \left[ \sup_{s \in [0,t]} |X_s-Z_s| \right]
\leq 
\widetilde{C} t \exp(L_b t) \eta,
\end{equation}
where $\widetilde{C}$ only depends on $C_\alpha$ and $L_b$. We eventually write that 
\begin{align*}
\E \left[ \sup_{s \in [0,t]} |X_s-Y_s| \right]
&=
\E \left[ \sup_{s \in [0,t]} |X_s-Z_s-e_s| \right]
\\
&\leq
\E \left[ \sup_{s \in [0,t]} |X_s-Z_s| \right]
+
\E \left[ \sup_{s \in [0,t]} |e_s| \right].
\end{align*}
From~\eqref{eq:titi2} and the above bound, we deduce the claimed bound~\eqref{eq:resu}.
\end{proof}

Referring to Remark~\ref{rem:holder}, if we assume~\eqref{eq:hyp_alpha_p} for some $1 \leq p < 2$ rather than~\eqref{eq:hyp_alpha}, the above proof can be modified to prove
\begin{equation}
\label{eq:resu_p}
\E \left[ \sup_{s \in [0,t]} |X_s-Y_s|^p \right]
\leq
C \left(\E \left[ \sup_{ s \in [0,t] } |e_s|^2 \right]\right)^{p/2}
\end{equation}
instead of~\eqref{eq:resu}. Indeed, from~\eqref{eq:holder}, we deduce that
$$
\sup_{s \in [0,t]} |X_s-Z_s|^p
\le 
2^p \exp(p L_b t) \left( \sup_{ s \in [0,t] } |e_s|^p \right) \left( \int_0^t \left( \alpha(|X_s|) + \alpha(|Y_s|) + L_b/2 \right) \, ds \right)^p,
$$
thus, using H\"older inequality with exponents $2/p$ and $2/(2-p)$, we get
\begin{align*}
&
\E \left[ \sup_{s \in [0,t]} |X_s-Z_s|^p \right]
\\
&\le
C_1 \left( \E \left[ \sup_{ s \in [0,t] } |e_s|^2 \right] \right)^{p/2} \ \left( \E \left[ \int_0^t \left( \alpha(|X_s|) + \alpha(|Y_s|) + L_b/2\right) \, ds \right]^{2p/(2-p)} \right)^{1-p/2}
\\
&\le
C_2 \eta^p \ \left( \E \left[ \int_0^t \Big[ \Big( \alpha(|X_s|) \Big)^{2p/(2-p)} + \Big( \alpha(|Y_s|) \Big)^{2p/(2-p)} + 1 \Big] \, ds \right] \right)^{1-p/2},
\end{align*}
where, as above, $\eta=\sqrt{\E \left[ \sup_{ s \in [0,t] } |e_s|^2 \right]}$ and $C_1$ and $C_2$ are constants only depending on $p$, $t$ and $L_b$. Since the law of $X_s$ and $Y_s$ is the same and does not depend on time, 
$$
\E\left[ \alpha(|X_s|)^{2p/(2-p)}\right] = \E\left[ \alpha(|Y_s|)^{2p/(2-p)}\right] = \E\left[ \alpha(|X_0|)^{2p/(2-p)}\right] = C_{\alpha,p} < \infty
$$
in view of~\eqref{eq:hyp_alpha_p}. We hence have
$\dps \E \left[ \sup_{s \in [0,t]} |X_s-Z_s|^p \right]
\leq 
C \eta^p,
$
where $C$ only depends on $p$, $t$, $C_{\alpha,p}$ and $L_b$. From this estimate and the fact that $X_s-Y_s=X_s-Z_s-e_s$, we deduce~\eqref{eq:resu_p}. 


\section{Proof of Proposition~\ref{prop:main_global}}\label{sec:conc}

In this section, we complete the proof of the error estimate~\eqref{eq:main_result} between the effective dynamics and the original dynamics, by combining the estimate on the last term in~\eqref{eq:diff} obtained in Section~\ref{sec:LZ} (namely~\eqref{eq:bb}) together with the Gronwall-type argument of Section~\ref{sec:gronwall} (i.e. estimate~\eqref{eq:resu}).

Recall that the exact dynamics satisfies (see~\eqref{eq:X1_exact})
$$
X^1_t - X^1_0 = - \int_0^t b(X^1_s) \, ds + \sqrt{2 \beta^{-1}} \, W^1_t + \int_0^t L^{X^1_s} u(X_s) \, ds
$$
whereas the effective dynamics satisfies (see~\eqref{eq:X1_approx})
$$
\xi_t - X^1_0 = - \int_0^t b(\xi_s) \, ds + \sqrt{2 \beta^{-1}} \, W^1_t.
$$
We set
$$
f_s = L^{X^1_s} u(X_s), \quad e_t = \int_0^t L^{X^1_s} u(X_s) \, ds,
$$
and we are thus in the setting of Lemma~\ref{lem:gronwall}, where $\xi_t$ satisfies a dynamics of the form~\eqref{eq:un} while $X^1_t$ satisfies a dynamics of the form~\eqref{eq:deux}.

In view of the assumptions~\eqref{eq:b_lip},~\eqref{eq:hyp_alpha} and~\eqref{eq:equi}, we see that the assumptions of Lemma~\ref{lem:gronwall} are satisfied. The bound~\eqref{eq:resu} thus yields that
\begin{equation}
\label{eq:bound_error}
\E \left[ \sup_{s \in [0,T]} |\xi_s-X^1_s| \right] \leq C \eta,
\end{equation}
where 
$$
\eta^2 
= 
\EE \left[ \sup_{0 \leq t \leq T} \left| e_t \right|^2 \right]
=
\EE \left[ \sup_{0 \leq t \leq T} \left| \int_0^t L^{X^1_s} u(X_s) \, ds\right|^2 \right]
\leq 
8 T \beta \frac{\kappa^2}{\rho^2},
$$
the last inequality being the main output of Section~\ref{sec:LZ} (see~\eqref{eq:bb} in Proposition~\ref{prop:bb}). Inserting the above estimate in~\eqref{eq:bound_error}, we deduce the claimed bound~\eqref{eq:main_result}. This concludes the proof of Proposition~\ref{prop:main_global}.

\begin{remark}
\label{rem:holder2}
Following Remark~\ref{rem:holder}, we now replace the assumption~\eqref{eq:hyp_alpha} by the stronger assumption~\eqref{eq:hyp_alpha_p} for some $1 \leq p < 2$. Following the same lines as in the proof of Proposition~\ref{prop:main_global}, and using~\eqref{eq:resu_p} instead of~\eqref{eq:resu}, we obtain
\begin{equation}
\label{eq:bound_error_p}
\E \left[ \sup_{s \in [0,T]} |\xi_s-X^1_s|^p \right] \leq C \eta^p,
\end{equation}
where, in view of~\eqref{eq:bb}, we have $\dps \eta^2 \leq 8 T \beta \frac{\kappa^2}{\rho^2}$. Inserting this estimate in~\eqref{eq:bound_error_p}, we deduce~\eqref{eq:main_result_p}.
\end{remark}

\section{Application: a quantitative averaging result}\label{sec:homog}

In this section, as an application of the techniques presented above, we show how to obtain a quantitative result on the error introduced by an averaging principle. We would like to stress that our result holds without assuming that the effective drift function is Lipschitz, which is to the best of our knowledge the assumption made in similar quantitative results that have been previously obtained in the literature.

\medskip

Let us consider the stochastic differential equation
\begin{equation}\label{eq:Xeps}
\left\{
\begin{aligned}
dX^{\varepsilon,1}_t &= -\partial_1 V(X^\varepsilon_t) \, dt + \sqrt{2 \beta^{-1}} \, dW^1 _t
\\
dX^{\varepsilon,i}_t &= -\frac{1}{\eps} \, \partial_i V(X^\varepsilon_t) \, dt + \sqrt{\frac{2 \beta^{-1}}{\varepsilon}} \, dW^i_t \qquad \text{ for $i=2,\ldots,n$}
\end{aligned}
\right.
\end{equation}
where $X^{\varepsilon,i}_t$ denotes the $i$-th component of the vector $X^\varepsilon_t \in \R^n$ and $\varepsilon$ is a positive constant. The initial condition $X^\varepsilon_0=X_0$ is assumed to be independent of $\varepsilon$ for simplicity. We note that $d\mu = \psi(x) \, dx$ is again the invariant measure of~\eqref{eq:Xeps}, and that it is independent of $\eps$.

By the averaging principle (see for example~\cite{pavliotis-stuart-07}), it is expected that, in the limit $\varepsilon \to 0$, the process $(X^{\varepsilon,1}_t)_{t \ge 0}$ converges to the process $(\xi_t)_{t \ge0}$ satisfying
\begin{equation}\label{eq:dyn_eff_eps}
d\xi_t = -b(\xi_t) \, dt + \sqrt{2 \beta^{-1}} dW^1_t
\end{equation}
with initial condition $\xi_0=X^1_0$, and where $b$ is defined by~\eqref{eq:def_b}. Using the techniques presented above, we are able to prove the following convergence result:
\begin{proposition}\label{prop:homog}
Assume that~\eqref{eq:poincare},~\eqref{eq:bound_d12V},~\eqref{eq:b_lip} and~\eqref{eq:hyp_alpha} hold, and that the system starts at equilibrium:
\begin{equation}
\label{eq:equi_eps}
X_0 \sim \mu.
\end{equation}
Consider $(X_t)_{0 \le t \le T}$ solution to~\eqref{eq:Xeps} and $(\xi_t)_{0 \le t \le T}$ solution to~\eqref{eq:dyn_eff_eps} over a bounded time interval $[0,T]$. Then, there exists a constant $C$, that is independent of $\varepsilon$, $\rho$ and $\kappa$, and that only depends on $T$, $C_\alpha(\beta)$ and $L_b$, such that
\begin{equation}
\label{eq:main_result_eps}
\EE \left( \sup_{0 \leq t \leq T} \left| X^{\varepsilon,1}_t - \xi_t \right| \right) 
\leq 
C \sqrt{\beta \varepsilon} \ \frac{\kappa}{\rho}.
\end{equation}
\end{proposition}

\begin{proof}
The proof consists in exactly following the same arguments as for the proof of Proposition~\ref{prop:main_global}, keeping track of the dependency of the constants on~$\varepsilon$. Let us emphasize the modifications in the various steps of the proof.

The equation~\eqref{eq:diff} still holds in our context:
\begin{equation}
\label{eq:diff_eps}
X^{\varepsilon,1}_t - \xi_t  
= 
\int_0^t \left( b(\xi_s) - b(X^{\varepsilon,1}_s) \right) \, ds + \int_0^t L^{X^{\varepsilon,1}_s} u(X^{\varepsilon}_s) \, ds,
\end{equation}
where, as before, $u$ is the solution to the Poisson problem~\eqref{eq:def_u_xi} using the same definition~\eqref{eq:def_Lxi} for the operator $L^\xi$. In particular, $u$ does not depend on~$\varepsilon$. By using the results of Sections~\ref{sec:f} and~\ref{sec:poisson}, we have (see Lemma~\ref{lem:poisson}) that
\begin{equation}
\label{eq:estim_wnu_eps}
\int_{\R^n} \left| \wn u \right|^2 \psi \le \beta^2 \, \frac{\kappa^2}{\rho^2}.
\end{equation}

The infinitesimal generator of the process $(X^\varepsilon_t)_{t \ge 0}$ is $L_\varepsilon$ defined by: for any smooth function $v:\R^n \to \R$, 
$$
L_\varepsilon v= - \partial_1 V \, \partial_1 v + \beta^{-1} \partial_{11} v + \frac{1}{\varepsilon} L^\xi v,
$$
to be compared with~\eqref{eq:def_L}.

The main modification in the proof of Proposition~\ref{prop:main_global} is in Lemma~\ref{lem:landim_olla}. Now, the estimate~\eqref{eq:resu:olla} is the following. For any function $\Phi: \R^n \to \R^n$ in $({\mathcal C}^\infty \cap L^2(\psi))^n$, we have
\begin{equation}
\label{eq:resu:olla_eps}
\EE \left[ \sup_{0 \leq t \leq T} \left| \int_0^t \nabla^\star \Phi(X_s) \, ds \right|^2 \right] 
\leq 
8 T \beta \,\left(\| \Phi_1 \|^2_{L^2(\psi)}+  \varepsilon \sum_{i=2}^n\| \Phi_i \|^2_{L^2(\psi)} \right)
\end{equation}
where $\Phi_i$ denotes the $i$-th component of $\Phi$ and, as above, $\nabla^\star \Phi = \beta \nabla V \cdot \Phi - {\rm div}(\Phi)$. Let us give a few details on how~\eqref{eq:resu:olla_eps} is obtained, mimicking the proof of Lemma~\ref{lem:landim_olla}. The equation $\eta w_{\varepsilon,\eta} - L_\varepsilon w_{\varepsilon,\eta} = - \nabla^\star \Phi$ (compare with~\eqref{eq:poisson_L_prime_eta}) has a unique solution $w_{\varepsilon,\eta} \in H^1(\psi)$. The associated variational formulation is: for any $v \in H^1(\psi)$,
$$
\eta \int_{\R^n} w_{\varepsilon,\eta} \, v \, \psi + \beta^{-1} \int_{\R^n} \left(\partial_1 w_{\varepsilon,\eta} \, \partial_1 v + \frac{1}{\varepsilon} \sum_{i=2}^n \partial_i w_{\varepsilon,\eta} \, \partial_i v \right) \, \psi 
= 
- \int_{\R^n} \Phi \cdot \nabla v \, \psi.
$$
Taking $v = w_{\varepsilon,\eta}$, we obtain the a priori estimate (compare with~\eqref{eq:a_priori}--\eqref{eq:a_priori2})
\begin{equation}
\label{eq:olla2bis_eps}
\| \partial_1 w_{\varepsilon,\eta} \|^2_{L^2(\psi)} + \frac{1}{\varepsilon} \sum_{i=2}^n \| \partial_i w_{\varepsilon,\eta} \|^2_{L^2(\psi)} 
\leq 
\beta^2 \left( \| \Phi_1 \|^2_{L^2(\psi)} + \varepsilon \sum_{i=2}^n \| \Phi_i \|^2_{L^2(\psi)} \right)
\end{equation}
and
\begin{equation}
\label{eq:olla2bisbis_eps}
\lim_{\eta \to 0} \eta \| w_{\varepsilon,\eta} \|_{L^2(\psi)} = 0.
\end{equation}
Following the Step 2 of the proof of Lemma~\ref{lem:landim_olla}, we write, by It\^o's calculus, that
\begin{multline*}
w_{\varepsilon,\eta}(X^\varepsilon_t) - w_{\varepsilon,\eta}(X_0) 
= 
\int_0^t L_\varepsilon w_{\varepsilon,\eta}(X^\varepsilon_s) \, ds
\\
+ \sqrt{2 \beta^{-1}} \int_0^t \left( \partial_1 w_{\varepsilon,\eta} (X^\varepsilon_s) dW^1_s + \frac{1}{\sqrt{\varepsilon}} \sum_{i=2}^n \partial_i w_{\varepsilon,\eta} (X^\varepsilon_s) dW^i_s\right),
\end{multline*}
which is to be compared with~\eqref{eq:alpha1_}. The process $(X^\varepsilon_s)_{0 \le s \le T}$ is still reversible, and thus, one obtains, using~\eqref{eq:olla2bis_eps}, that
\begin{eqnarray*}
\EE \left[ \sup_{0 \leq t \leq T} 
\left| \int_0^t L_\varepsilon w_{\varepsilon,\eta} (X^\varepsilon_s) \, ds \right|^2 \right] 
& \leq &
8 \beta^{-1} T \int_{\R^n} \left( (\partial_1 w_\varepsilon)^2 + \frac{1}{\varepsilon}\sum_{i=2}^n  (\partial_i w_\varepsilon)^2\right)\, \psi
\\
& \leq &
8 T \beta \,\left(\| \Phi_1 \|^2_{L^2(\psi)}+  \varepsilon \sum_{i=2}^n\| \Phi_i \|^2_{L^2(\psi)} \right),
\end{eqnarray*}
to be compared with~\eqref{eq:olla2}. Using~\eqref{eq:olla2bisbis_eps}, we obtain~\eqref{eq:resu:olla_eps} as in the proof of Lemma~\ref{lem:landim_olla}.

Choosing $\Phi=(0,\wn u)$ in~\eqref{eq:resu:olla_eps} (as in the proof of Proposition~\ref{prop:bb}), one obtains, using~\eqref{eq:estim_wnu_eps},
$$
\EE \left[ \sup_{0 \leq t \leq T} \left| \int_0^t L^{X^{\varepsilon,1}_s} u(X^\varepsilon_s) \, ds\right|^2 \right]
\leq 
8 T \beta \varepsilon \frac{\kappa^2}{\rho^2},
$$
which is to be compared with~\eqref{eq:bb}. This estimates gives the magnitude of last term in~\eqref{eq:diff_eps}. The end of the proof follows exactly the same lines as for Proposition~\ref{prop:main_global}.
\end{proof}

Let us make three comments on the previous result. First, the stationarity assumption~\eqref{eq:equi_eps} can be weakened in a similar way as in Corollary~\ref{coro:main_global} above by using the same conditioning argument. 

Second, it is easy to generalize the previous result to the situation where each component of $(X_t)_{t \ge 0}$ is scaled in time with a parameter $\varepsilon_i$: for $i=2,\ldots,n$,
$$
dX^i_t = -\frac{1}{\eps_i} \partial_i V(X_t) \, dt + \sqrt{\frac{2\beta^{-1}}{\varepsilon_i}} \, dW^i_t
$$
while we keep, for the first component,
$$
dX^1_t = - \partial_1 V(X_t) \, dt + \sqrt{2\beta^{-1}} \, dW^1_t.
$$
In this case, one obtains a similar estimate as~\eqref{eq:main_result_eps}, which $\varepsilon$ being replaced by $\max(\varepsilon_2, \ldots, \varepsilon_n)$.

Third, we notice that using the simple approach of Section~\ref{sec:simple_csq} on the stochastic differential equation~\eqref{eq:Xeps}, one obtains the estimate~\eqref{eq:main_result_bis_global} with an upper bound independent of $\varepsilon$ (recall indeed that $b$, and thus $f$ defined by~\eqref{eq:def_f}, are independent of $\eps$). This shows the interest of the approach developed in Sections~\ref{sec:LZ} and~\ref{sec:gronwall}.

\section*{Acknowledgments} 

S. Olla acknowledges support by the ANR LSD. The work of F. Legoll and T. Leli\`evre is supported by the European Research Council under the European Union's Seventh Framework Programme (FP/2007-2013) / ERC Grant Agreement number 614492. T. Leli\`evre would like to thank Dirk Bl\"omker (Universit\"at Augsburg) for useful discussions on a preliminary version of this work.

\bibliography{eff_dyn_legoll_lelievre_olla.bib}

\begin{thebibliography}{10}

\bibitem{ABC-00}
C.~An\'e, S.~Blach\`ere, D.~Chafa\"i, P.~Foug\`eres, I.~Gentil, F.~Malrieu,
  C.~Roberto, and G.~Scheffer.
\newblock {\em Sur les in\'{e}galit\'{e}s de Sobolev logarithmiques}.
\newblock Soci\'{e}t\'{e} Math\'{e}matique de France, 2000.
\newblock In French.

\bibitem{e-vanden-eijnden-04}
W.~E and E.~Vanden-Eijnden.
\newblock Metastability, conformation dynamics, and transition pathways in
  complex systems.
\newblock In {\em Multiscale modelling and simulation}, volume~39 of {\em Lect.
  Notes Comput. Sci. Eng.}, pages 35--68. Springer, Berlin, 2004.

\bibitem{givon-kupferman-stuart-04}
D.~Givon, R.~Kupferman, and A.~Stuart.
\newblock Extracting macroscopic dynamics: model problems and algorithms.
\newblock {\em Nonlinearity}, 17(6):R55--R127, 2004.

\bibitem{grunewald-otto-villani-westdickenberg-09}
N.~Grunewald, F.~Otto, C.~Villani, and M.G. Westdickenberg.
\newblock A two-scale approach to logarithmic {S}obolev inequalities and the
  hydrodynamic limit.
\newblock {\em Ann. Inst. H. Poincar\'e Probab. Statist.}, 45(2):302--351,
  2009.

\bibitem{gyongy-86}
I.~Gyongy.
\newblock Mimicking the one-dimensional marginal distributions of processes
  having an {I}t{\^o} differential.
\newblock {\em Probab. Th. Rel. Fields}, 71:501--516, 1986.

\bibitem{hanggi-talkner-barkovec-90}
P.~H\"anggi, P.~Talkner, and M.~Borkovec.
\newblock Reaction-rate theory: fifty years after {K}ramers.
\newblock {\em Reviews of Modern Physics}, 62(2):251--342, 1990.

\bibitem{komorowski-landim-olla-12}
T.~Komorowski, C.~Landim, and S.~Olla.
\newblock {\em Fluctuations in Markov Processes, Time Symmetry and Martingale
  Approximation}, volume 345 of {\em Grundlheren der Mathematischen
  Wissenschaften}.
\newblock Springer, Berlin, New York, 2012.

\bibitem{legoll-lelievre-10}
F.~Legoll and T.~Leli{\`e}vre.
\newblock Effective dynamics using conditional expectations.
\newblock {\em Nonlinearity}, 23:2131--2163, 2010.

\bibitem{legoll-lelievre-12}
F.~Legoll and T.~Leli{\`e}vre.
\newblock Some remarks on free energy and coarse-graining.
\newblock In B.~Engquist, O.~Runborg, and R.~Tsai, editors, {\em Multiscale
  Modeling and Simulation in Science}, volume~82 of {\em Lecture Notes in
  Computational Science and Engineering}, pages 279--329. Springer, 2012.

\bibitem{lelievre-09}
T.~Leli{\`e}vre.
\newblock A general two-scale criteria for logarithmic {S}obolev inequalities.
\newblock {\em J. Funct. Anal.}, 256(7):2211--2221, 2009.

\bibitem{lelievre-rousset-stoltz-08}
T.~Leli\`evre, M.~Rousset, and G.~Stoltz.
\newblock Long-time convergence of an adaptive biasing force method.
\newblock {\em Nonlinearity}, 21:1155--1181, 2008.

\bibitem{lelievre-rousset-stoltz-book-10}
T.~Leli\`evre, M.~Rousset, and G.~Stoltz.
\newblock {\em Free energy computations: A mathematical perspective}.
\newblock Imperial College Press, 2010.

\bibitem{lyons-zhang-94}
T.J. Lyons and T.S. Zhang.
\newblock Decomposition of {D}irichlet processes and its application.
\newblock {\em The Annals of Probability}, 22(1):494--524, 1994.

\bibitem{maragliano-fischer-vanden-einjden-ciccotti-06}
L.~Maragliano, A.~Fischer, E.~Vanden-Eijnden, and G.~Ciccotti.
\newblock String method in collective variables: minimum free energy paths and
  isocommittor surfaces.
\newblock {\em J. Chem. Phys.}, 125:024106, 2006.

\bibitem{pavliotis-stuart-07}
G.A. Pavliotis and A.M. Stuart.
\newblock {\em Multiscale methods: averaging and homogenization}.
\newblock Springer, 2007.

\end{thebibliography}
\bibliographystyle{plain}

\end{document}